\newfont{\Bbb}{msbm10 scaled\magstephalf}
\documentclass[12pt,reqno]{amsart}
\usepackage{relsize}
\usepackage{amsfonts}
\usepackage{amsmath, amsthm, amscd, amsfonts}
\usepackage{amssymb}
\usepackage{tikz}
\usepackage{amsmath}
\usepackage{tikz-cd}
\usepackage{xcolor}
\usepackage{amscd}
\usepackage{graphicx}
 \usepackage{epstopdf}
\allowdisplaybreaks[4]
 \newtheorem{thm}{Theorem}[section]
 \newtheorem{cor}[thm]{Corollary}
 \newtheorem{lem}[thm]{Lemma}
 \newtheorem{prop}[thm]{Proposition}
 \theoremstyle{definition}
 \newtheorem{defn}[thm]{Definition}
\theoremstyle{remark}
 \newtheorem{rem}[thm]{Remark}
 \newtheorem{exm}[thm]{Example}
 \numberwithin{equation}{section}
\newcommand{\pf}{\begin{proof}}
\newcommand{\zb}{\end{proof}}
\newcommand{\la}{\langle}
\newcommand{\ra}{\rangle}
\newcommand{\ol}{\overline}
\newcommand{\ma}{\mathcal}

\def\dim{\mathop{\rm dim}\nolimits}

\begin{document}
\title[Spectra and invariant subspaces]{Spectra and invariant subspaces of compressed shifts on nearly invariant subspaces}
\author[Y. Liang]{Yuxia Liang}
\address{Yuxia Liang \newline School of Mathematical Sciences,
Tianjin Normal University, Tianjin 300387, P.R. China.} \email{liangyx1986@126.com}
\author[J. R. Partington]{Jonathan R. Partington}
\address{Jonathan R. Partington \newline School of Mathematics,
  University of Leeds, Leeds LS2 9JT, United Kingdom.}
 \email{J.R.Partington@leeds.ac.uk}

\begin{abstract} While the spectral properties and invariant subspaces of compressed shifts on model spaces are well understood, their behaviour on nearly $S^*$-invariant subspaces, a natural generalization under weaker structural constraints, remains largely unexplored. These operators are closely related to the Clark-type unitary operators, yet differ from them in several ways. In this paper, we completely characterize the point spectrum, whole spectrum and invariant subspace structure for such compressed shifts by  unitary equivalence, using the Frostman shift, Crofoot transform, and Sz.-Nagy--Foias theory.  The results reveal how the relaxation of $S^*$-invariance impacts spectral structure and invariant subspaces, bridging a gap between classical model space theory and broader function-theoretic settings.
\end{abstract}

\maketitle

\section{Preliminaries}

The compressed shift  is a central object in operator theory, arising naturally in the study of Hilbert space contractions. For an inner function $\theta$, the compressed shift $S_\theta$ on the model space $K_\theta := H^2 \ominus \theta H^2$ is defined as the compression of the unilateral shift $S \colon f(z) \mapsto zf(z)$ on the Hardy space $H^2:=H^2(\mathbb{D})$ by $$
S_\theta := P_\theta S|_{K_\theta},$$
where $P_\theta$ is the orthogonal projection of $L^2$ onto the model space $K_\theta.$ It is defined by $$(P_\theta f)(\lambda)=\langle f,k_\lambda^\theta\rangle,\;\lambda \in \mathbb{D}, f\in L^2,$$ and $k_\lambda^\theta$ is the reproducing kernel for $K_\theta,$ expressible in terms of the standard $H^2$ kernel  $k_\lambda(z)$  as \begin{align*}k_\lambda^{\theta}(z)=
 \frac{1-\ol{\theta(\lambda)}\theta(z)}
 {1-\ol{\lambda}z}=(1-\ol{\theta(\lambda)}\theta(z))
 k_\lambda(z).\end{align*}  The operator  $S_\theta$ highlights the deep relationship between function theory and operator theory, with its structure reflecting the analytic properties of $\theta$. Moreover, it  serves as a prototype for all \emph{completely non-unitary} (c.n.u.)  contractions  on Hilbert spaces in the Sz.-Nagy--Foias functional model. Specifically, a bounded linear operator $T$ on a Hilbert space $\ma{H}$ is defined as a contraction if $\|Tf\|\leq \|f\|$ for all $f\in \ma{H}.$ A contraction $T$ is called  c.n.u.  if there is no nontrivial $T$-reducing subspace $\ma{N}\subseteq \ma{H}$  such that $T|_{\ma{N}}$ is unitary.

 The study of invariant subspaces for linear operators has led to  profound results, bridging  complex function theory, operator theory, and functional analysis. A cornerstone in this field is the Beurling theorem, which characterizes   shift-invariant subspaces of the Hardy space \( H^2\) (see, e.g.  \cite[A.~Corollary 1.4.1]{Nik}).
 To extend these ideas to Hardy spaces of an annulus, Hitt introduced the concept of nearly $S^*$-invariant subspaces in \cite{hitt},   generalizing Hayashi's work on the kernels
 of Toeplitz operators \cite{Ha}. Further contributions by Sarason \cite{Sa1, Sa2} deepened the connections with Toeplitz   kernels, further enriching the theory.

 A closed subspace \(\mathcal{M} \subseteq H^2 \) is called nearly \(S^*\)-invariant  (originally weakly invariant) if, for any \(f \in \mathcal{M}\) with \(f(0) = 0\), it follows that \(S^* f \in \mathcal{M}\). Intuitively, \(\mathcal{M} \) is nearly \(S^*\)-invariant if the zeros of its functions  can be factored out without leaving the space. For the broader classes of operators, we introduce the notion of nearly invariant subspaces, first for a left invertible operator (see \cite{LP2}) and later expand to general $C_0$-semigroups $\{T(t)\}_{t\geq 0}$ on a separable infinite-dimensional Hilbert space (see \cite{LP3}).  In the nearly $S^*$-invariant setting, Hitt's seminal work provides a complete characterization of such subspaces in \(H^2 \) as below.

\begin{thm}\cite[Proposition 3]{hitt} \label{thm Hitt}The nearly $S^*$-invariant subspaces of $H^2 $ have the form $\ma{M}=hK$, with $h\in \ma{M}$ of unit norm, $h(0)>0,$  $h$ orthogonal to all elements of $\ma{M}$ vanishing at the origin, $K$  an $S^*$-invariant subspace, and the operator of multiplication by $h$  isometric from $K$ into $H^2$.
\end{thm}
Later this is generalized to a vectorial setting in \cite[Theorem 4.4]{CCP10}. More generally, it is further  extended to nearly $S^*$-invariant subspaces with defect $m$ (i.e., if $f(0)=0$ then $S^*f \in \mathcal{M}+ \mathcal{F}$ for a  fixed subspace $\mathcal{F}$ of minimal dimension $m$)  in \cite{CGP}, where the closely related concept  of almost-invariant subspace  is also investigated, along with its vectorial version in \cite{CDP20}.

The function $h$ in Hitt's theorem is called the \emph{extremal function} for $\ma{M}$ due to its being the unique solution to the extremal problem $$\sup\{\mbox{Re}\; g(0):\;g\in \ma{M},\;\|g\|=1\},$$ and then
either $K$ is trivial or there is an inner function $\theta$ with $\theta(0)=0$ such that $\ma{M}=hK_\theta,$ which is a generalization of $S^*$-invariant subspaces.

As shown in \cite[Proposition 9.14]{GMR}, the invariant subspaces of $S_\theta$ in $K_\theta$ are precisely $K_\theta \ominus K_\phi$, where $\phi$ divides $\theta$.
This classical result motivates us to ask a natural question as below.

\begin{center}
  \textbf{ Question 1}: \textit{How does this characterization generalize when $K_\theta$ is replaced by a nearly $S^*$-invariant subspace $\ma{M}= h K_\theta$?}
\end{center}

We completely answer \textbf{Question 1} in Theorem \ref{thm com} and characterize the spectrum of this compressed shift in Theorem \ref{thm whole}; both results are illustrated in Example \ref{exm si}.
 \vspace{0.1mm}

To analyze the compressed shift on a nearly $S^*$-invariant subspace $\ma{M} = h K_\theta$, we first revisit the classical compressed shift $S_\theta$ on the model space $K_\theta$. Recall that $S_\theta$, defined as the compression of $S$ to $K_\theta,$ is the prototypical example of a truncated Toeplitz operator, introduced by Sarason to study function-theoretic properties of model spaces.  Given a symbol $g\in L^2$, the truncated Toeplitz operator $A_g^\theta$ is densely defined on    bounded $f\in K_\theta$ by
\begin{align*}A_g^\theta f=P_{\theta}(g  f),\;\mbox{for}\;f\in \mbox{Dom}(A_g^\theta),\end{align*} with the domain $\mbox{Dom}(A_g^\theta)=\{f\in K_\theta:\; g f\in L^2\}.$  These operators share many characteristics with classical Toeplitz operators. However, they also exhibit striking differences
(see, e.g. \cite{Sarason1}).  For example, there are bounded truncated Toeplitz operators with unbounded symbols (though any truncated Toeplitz operator with a bounded symbol is itself bounded). Especially, \cite[Theorem 3.1]{Sarason1} implies $A_g^{\theta}=0$ if and only if $g\in \theta H^2+\ol{\theta H^2}.$

Given a nearly $S^*$-invariant subspace $\ma{M}=hK_\theta$  with extremal function $h$ and associated inner function $\theta$, $\theta(0)=0$, \cite[Lemma 2.5]{HR} identified the orthogonal projection $P_\ma{M}$ from $L^2$ onto $\ma{M}$ as $$P_\ma{M} f=hP_\theta(\ol{h} f), \;\;f\in L^2.$$ As a corollary, \cite[Corollary 2.6]{HR} exhibits the reproducing kernel for  $\ma{M}=hK_\theta$
as given by
 \begin{align*}k_{\lambda}^\ma{M}(z):=\ol{h(\lambda)}h(z) \frac{1-\ol{\theta(\lambda)}\theta(z)}{1-\ol{\lambda} z}=\ol{h(\lambda)}h(z)k_\lambda^{\theta}(z).\end{align*} So, given a measurable function $g$ on $\mathbb{T}$, the truncated Toeplitz operator $A_g^\ma{M}$ on a nearly $S^*$-invariant subspace $\ma{M} = hK_\theta$ is  defined as \begin{align*}A_g^\ma{M} f=P_\ma{M}(g  f)=hP_\theta(\ol{h}gf),\; f\in \mbox{Dom}(A_g^\ma{M}),\end{align*} where Dom$(A_g^\ma{M})=\{f \in \ma{M}:\; g f \in L^2\}.$

The multiplier $M_h: K_\theta \rightarrow  hK_\theta$ is a surjective isometry as in Theorem \ref{thm Hitt}, and so it is unitary,  with adjoint $M_h^*= M_{h^{-1}},$ and so there follows the
  result of \cite[Theorem 3.1]{HR}.

\begin{thm}\label{Mtheta} Let $A_g^\ma{M}$ be a bounded operator on $\ma{M}=hK_\theta$ with extremal function $h$ and associated inner function $\theta$, $\theta(0)=0$.  Then there is a unitary equivalence between $A^\ma{M}_g$ and $A^\theta_{|h|^2g}$ given by $A_g^\ma{M} = M_hA^\theta_{|h|^2g}M_h^*.$\end{thm}

\begin{proof} Taking $f \in \ma{M}=hK_\theta$, it holds that $$A^\ma{M}_gf = hP_\theta (\ol{h}gf) = M_hA^\theta_{|h|^2 g}(h^{-1} f)= M_h A_{|h|^2g}^\theta M_h^* f,$$ ending the proof.\end{proof}

We note that Theorem \ref{Mtheta} and \cite[Theorem 3.1]{Sarason1} provide a characterization of  zero truncated Toeplitz operators on $\ma{M}=hK_\theta$.

\begin{cor} Let $\ma{M}=hK_\theta$ with extremal function $h$ and associated inner function $\theta$, $\theta(0)=0$ and $|h|^2g\in L^2,$ then $A_g^\ma{M}=0$ on $\ma{M}$  if and only if $|h|^2g$ belongs to $\theta H^2+\ol{\theta}\ol{H^2}.$
\end{cor}

Kernel spaces are crucial for understanding the structure of linear operators, especially Toeplitz operators and truncated Toeplitz operators (see, e.g. \cite{CaP,Sarason1}).  \cite[Corollary 3.2]{Lou} (with $n=1$) states that the kernel of a truncated Toeplitz operator on model space is nearly $S^*$-invariant with defect at most $1$. This result can be extended to the operator $A^\ma{M}_g$ acting on a nearly $S^*$-invariant subspace $\ma{M}=hK_\theta$.

\begin{cor} Let $A^\ma{M}_g$ be a bounded truncated Toeplitz operator on a nearly $S^*$-invariant space $\ma{M} = hK_\theta \neq \{0\}$ with extremal function $h$ and associated inner function $\theta$, $\theta(0)=0$. Then $\ker A^\ma{M}_g$
is nearly $S^*$-invariant with defect at most $1$. \end{cor}

\begin{proof} Suppose that $f\in \ker A^\ma{M}_g$ and $f(0)=0$. We have $f=hk$ with some $k \in K_\theta$ and we note that by near $S^*$-invariance, it follows that  $h(0)\neq 0$. Otherwise every
function in $\ma{M}$ would vanish at $0$, which yields  $\ma{M} = \{0\}$. So it follows that $k(0) = 0.$  From $A^\ma{M}_gf = 0$, Theorem \ref{Mtheta} implies  $$M_hA_{|h|^2g}^\theta(M_h^*f)=M_hA_{|h|^2g}^\theta(h^{-1}f)=M_hA_{|h|^2g}^\theta(k)=0$$ and so $k\in \ker A^\theta_{|h|^2g}$. Hence  \cite[Corollary 3.2]{Lou} asserts that there is a one-dimensional subspace $\ma{F}=\langle x\rangle$, independent of $k$, such
that $ S^*k=\ol{z}k \in \ker A^\theta_{|h|^2g} +\ma{F}$.  Then there is a $\lambda \in \mathbb{C}$ such that $A^\theta_{|h|^2g}(\ol{z}k-\lambda x)=0.$  We can express this as  $$h^{-1} A_g^\ma{M} h(\ol{z}k-\lambda x)=h^{-1} A_g^\ma{M} (\ol{z}f-\lambda hx)=0.$$
That is, $S^*f-\lambda (hx) \in \ker A^\ma{M}_g$, as required.
\end{proof}
Finally, we demonstrate that the truncated Toeplitz operator $A_g^\ma{M}$ on a nearly $S^*$-invariant subspace  $\ma{M}=hK_\theta$ is $D$-symmetric with respect to the conjugation $D:=M_hCM_h^*$,  where  \begin{align}Cf=\theta \ol{z f},\label{Con} \end{align} is a bijective  map from $\theta H^2$ to $\ol{zH^2}$ which
also takes $K_\theta$ to itself.

\begin{rem}  \label{rem D-sym} Let $A_g^\ma{M}$ be a bounded operator on  $\ma{M}=hK_\theta$ with extremal function $h$ and associated inner function $\theta$, $\theta(0)=0$, then $ D A_g^\ma{M} D =  A_{\ol{g}}^\ma{M}$ with $D=M_hCM_h^*.$\end{rem}
\begin{proof} Theorem \ref{Mtheta} implies $M_h^* A_g^\ma{M} M_h = A^\theta_{|h|^2g},$ which further yields that $$ CM_h^* A_g^\ma{M} M_hC = CA^\theta_{|h|^2g}C=A^\theta_{|h|^2\ol{g}}.$$
Note $M_h^* A_{\ol{g}}^\ma{M} M_h = A^\theta_{|h|^2\ol{g}},$  and so
$$ CM_h^* A_g^\ma{M} M_hC = CA^\theta_{|h|^2g}C=M_h^* A_{\ol{g}}^\ma{M} M_h.$$
That is to say,
 $ DA_g^\ma{M} D=(M_hCM_h^*) A_g^\ma{M} (M_hCM_h^*) = A_{\ol{g}}^\ma{M}.$ \end{proof}

Building on the above foundational properties of $A_g^{\ma{M}}$ on $\ma{M}=hK_\theta$, we now turn to address Question 1. In the specific case where $g(z)=z,$   Theorem \ref{Mtheta} establishes that the compressed shift $A_z^\ma{M}$ on  $\ma{M}=hK_\theta$ satisfies the following key property
\begin{align} A_z^\ma{M}  =M_h A_{|h|^2 z}^\theta M_h^*.\label{SM} \end{align} Let $(A_z^\ma{M})^* $ denote the adjoint of $A_z^\ma{M}$, then for $f, g\in \ma{M},$ it holds that $$\langle (A_z^\ma{M})^* f,g\rangle=\langle f, A_z^\ma{M} g\rangle=\langle f,P_\ma{M}(Sg)\rangle=\langle f, Sg \rangle=\langle P_\ma{M}( \ol{z}f), g\rangle,$$ yielding  that     $ (A_z^\ma{M})^*=A_{\ol{z}}^\ma{M}.$ Meanwhile, it holds  that
\begin{align}(A_z^\ma{M})^*   =M_h A_{|h|^2\ol{z}}^\theta M_h^*.\label{SM*}\end{align}

Equations \eqref{SM} and \eqref{SM*} reveal that the spectral properties and invariant subspaces of the compressed shift $A_z^\ma{M}$ (respectively, its adjoint $(A_z^\ma{M})^*$)  on $\ma{M}=hK_\theta$ are fully determined by those of the truncated Toeplitz operator $A_{|h|^2z}^\theta$ (respectively,   $A_{|h|^2\bar{z}}^\theta$) on $K_\theta$, where $\theta(0)=0$. Specifically, the spectrum of $A_z^\ma{M}$ (respectively, $(A_z^\ma{M})^*$) coincides with that of $A_{|h|^2z}^\theta$ (respectively,  $A_{|h|^2\bar{z}}^\theta$) i.e. $\sigma(A_z^\ma{M})=\sigma(A_{|h|^2z}^\theta)$ (respectively, $\sigma(A_{\ol{z}}^\ma{M})=\sigma(A_{|h|^2\ol{z}}^\theta)$). Meanwhile, every invariant subspace  of $A_z^\ma{M}$ (respectively, $(A_z^\ma{M})^*$) is of the form $M_h(\ma{N})$, where $\ma{N}\subseteq K_\theta$ is an invariant subspace of $A_{|h|^2z}^\theta$ (respectively, $A_{|h|^2\bar{z}}^\theta$ ).\\

Driven by {\bf Question 1}, we investigate both the spectral properties and the invariant subspace problem for the   compressed shift operator arising in \eqref{SM}. This operator is unitarily equivalent to the truncated Toeplitz operator $A_v:=A_{|h|^2 z}^\theta$,  a rank-1 perturbation of $S_\theta$ (see Proposition \ref{Aztheta}). Unlike the Clark unitary operator $U_\lambda$ with $|\lambda|=1$ in \cite{Clark}, however, $A_v$ exhibits different behavior due to the condition $|v|=|\la \theta, |h|^2\ra|<1$. So new technical tools and methods are required for characterizing the relevant properties of $A_v$.

 The construction developed in this paper proceeds as follows. First, we derive explicit representations for the truncated Toeplitz operator $A_{|h|^2z}^\theta$ and its adjoint $A_{|h|^2\bar{z}}^\theta$ on the model space $ K_\theta$ with $\theta(0)=0$  in Propositions \ref{Aztheta} and \ref{Aztheta*}, establishing fundamental results on their cyclic vectors and actions on kernel functions in Propositions \ref{cyclic} and \ref{prop ker}. Next, we apply unitary equivalence, using the Frostman shift, Crofoot transform, and  Sz.-Nagy--Foias theory to rigorously analyze the spectral properties of $A_z^\ma{M}$, identifying its essential spectrum in Corollary \ref{cor ess},  point spectrum in Theorem \ref{thm ps} and whole spectrum in Theorem \ref{thm whole}.  Finally, we explore the generalized eigenvectors for any eigenvalue of $A_{|h|^2\bar{z}}^\theta$ with multiplicity $n \geq 2$ to display more finite-dimensional invariant subspaces in Proposition \ref{prop ge},  and then classify all invariant subspaces for   $A_z^\ma{M}$ acting on an arbitrary nearly $S^*$-invariant subspace $\ma{M}=hK_\theta$ in Theorem \ref{thm com}. Especially, we demonstrate Theorems \ref{thm whole} and  \ref{thm com} in Example \ref{exm si},  unifying these results within a generalized framework of model space theory.

\section{The compressed shifts on nearly invariant subspaces}

Given an inner function $\theta$ with $\theta(0)=0,$ there are two  orthogonal direct sums  \begin{align}K_\theta=K_z\oplus zK_{\theta\ol{z}}=K_{\theta\ol{z}}\oplus \mathbb{C}\theta\ol{z},\label{decom}\end{align}
and the backward shift $S^*$ acts as $0\oplus M_{\ol{z}}$ with respect to the first sum. Next $$\la A_{|h|^2\bar{z}}^\theta f, g\ra=\langle |h|^2\ol{z}f, g\rangle=\langle h\ol{z}f, hg\rangle=\langle \bar{z}f, g\rangle$$
if $\ol{z}f\in K_\theta,$ i.e., if $f\in zK_{\theta \bar{z}}$ and $g\in K_\theta$ is arbitrary. That is, on the subspace $zK_{\theta \bar{z}}$ these two operators coincide. Then we need to work out $A_{|h|^2\bar{z}}^\theta 1$ to understand $A_{|h|^2\bar{z}}^\theta$ on $K_\theta.$ Indeed,
\begin{align}\la A_{|h|^2\bar{z}}^\theta 1, g\ra= \langle |h|^2\bar{z}1, g\rangle=\langle h1, hgz\rangle=\langle 1, gz\rangle=0 \label{hg1}\end{align}
if $gz \in K_\theta.$ The second decomposition in \eqref{decom} together with \eqref{hg1}
show that    $A_{|h|^2 \bar{z}}^\theta 1 \in \mathbb{C}\theta {\bar z}$. That is, $A_{|h|^2 \bar{z}}^\theta 1$ is a multiple of $\theta \bar{z}$. Considering $\langle |h|^2\bar{z}1,\theta \bar{z}\rangle=\la |h|^2, \theta\ra,$ we finally conclude that
\begin{align} A_{|h|^2\bar{z}}^\theta 1=\bar{z}\theta \langle |h|^2, \theta\rangle.\label{AZ1} \end{align}

Thus we have established the following proposition.
\begin{prop}\label{Aztheta*} Let $B_v:=A_{|h|^2 \bar{z}}^\theta$ be the adjoint truncated Toeplitz operator on  a model space $K_\theta$ with  inner function $\theta,$ $\theta(0)=0$ and $v=\langle  \theta,|h|^2\rangle$. Then
 \begin{align}B_v f=\left\{ \begin{array}{ll}    \bar{z} f, &  f\in zK_{\theta \ol{z}}, \\ \ol{v}\theta \ol{z}, & f=1.  \end{array} \right.\label{AHZ*}\end{align}
 \end{prop}

Under the first decomposition in \eqref{decom}, \eqref{AHZ*} yields $B_v$ as a rank-1 perturbation of $S_\theta^*$ on $K_\theta.$ Moreover, since  $1\in K_\theta$ with $\theta(0)=0$, it holds $\la h, h\ra=\la 1, 1\ra=1$, so that $\|h\|=1.$  By Cauchy-Schwartz inequality, we deduce
\[
|v|=\bigl|\langle \theta, |h|^2 \rangle\bigr| = \bigl|\langle \theta h, h \rangle\bigr| \le \|\theta h\| \|h\| \le \|h\| \|h\| = 1
\]
and \( |\langle \theta h, h \rangle| = 1 \) if and only if \( \theta h = \beta h \) for some complex number \( \beta \). That is \( (\theta - \beta)h = 0 \) which is impossible, so \( v=\bigl|\langle \theta, |h|^2 \rangle\bigr| < 1 \).

Here we include an example  with inner $\theta$ having repeated zeros to exhibit the difference between $B_v$ and $S_\theta^*$ acting on $K_\theta.$ Our convention is that the $j$-th column in the matrix contains the image of the $j$-th vector, so if $S^* v_j =\sum_{k=1}^n c_k v_k,$ then the coefficients $c_1, \cdots, c_n$ are put in the $j$-th column.

\begin{exm} Let $\theta(z)=z^4$ ($z^n$ is similar). With respect to the orthonormal basis $\{1, z, z^2, z^3\}$ of $K_\theta$ we have the matrix
\begin{align*} [S_\theta^*]=\left(\begin{array}{cccc}0 &1& 0& 0\\
0& 0& 1& 0\\ 0& 0& 0& 1\\ 0 &0& 0& 0\end{array}\right).\end{align*}
The only eigenvalue is $0$ and the invariant subspaces are model subspaces of $K_\theta$, i.e., $\{0\}, K_z, K_{z^2}, K_{z^3}$ and
$K_{z^4}$. On the other hand, the matrix of $B_v$ is
$$[B_v] = \left(\begin{array}{cccc}0 & 1 & 0& 0\\0& 0& 1 &0\\0& 0 &0& 1\\ \alpha & 0& 0& 0\end{array}\right)$$
where $\alpha=\la |h|^2, z^4\ra$. Note that $[B_v]^4 = \alpha I$ and there are four distinct eigenvalues, the fourth roots
of $\alpha$ (in the case $ \alpha\neq 0$). So the lattice of invariant subspaces is rather different, being the linear spans of eigenvectors.
Note that the two lattices of invariant subspaces are very different, even though the matrices  differ in only one entry.\end{exm}

  We now analyse the action of $A_{|h|^2 z}^\theta$  on $K_\theta$ when $\theta(0)=0.$  Considering $M_h: K_\theta \rightarrow  hK_\theta$ is unitary in Theorem \ref{thm Hitt},  we obtain $$\la hf, hg\ra=\la f,g\ra,\;\mbox{for}\;f, g\in K_\theta. $$ Next, we deduce $$\la A_{|h|^2 z}^\theta f, g\ra=\la  hzf, hg\ra,\;\mbox{for}\; f,g\in K_\theta. $$ Thus, if $zf\in K_\theta$ we conclude that $ A_{|h|^2 z}^\theta f=zf$ by taking $g$ as the reproducing kernel. Employing the second orthogonal decomposition in \eqref{decom}, it follows that $zf\in K_\theta$ for all $f\in K_{\theta \bar{z}}.$ That means $A_{|h|^2z}^\theta$ and $S$ coincide on the subspace  $K_{\theta \bar{z}}.$

After that we use conjugation $C$ in \eqref{Con} to compute $A_{|h|^2 z}^\theta (\theta \bar{z})$. Since $1\in K_\theta$ due to $\theta(0)=0,$ $C1=\theta\bar{z}$, \eqref{AZ1} yields that \begin{align*} A_{|h|^2 z}^\theta (\theta \bar{z})&=A_{|h|^2 z}^\theta C1=C B_v 1=C(\bar{z}\theta \la |h|^2, \theta\ra)= \theta \ol{z \theta \la |h|^2, \theta\ra}z=v. \end{align*}  Hence we summarize the above calculations in the following proposition.
\begin{prop}\label{Aztheta} Let $A_v:=A_{|h|^2z}^\theta$ be the truncated Toeplitz operator on
a model space $K_\theta$ with  inner function $\theta,$ $\theta(0)=0$ and $v=\la \theta, |h|^2\ra$. Then
 \begin{align}A_v f=\left\{ \begin{array}{ll}   zf, &  f\in K_{\theta \ol{z}}, \\ v, & f= \theta \ol{z}.  \end{array} \right.\label{AHZ}\end{align}
 \end{prop}
Equation \eqref{AHZ} indicates how $A_v$ is a rank-1 perturbation of $S_\theta$ on $K_\theta.$  For $S_\theta$ on $K_\theta$, \cite[Proposition 9.13]{GMR} shows that it is a cyclic operator with cyclic vector $k_0=1-\ol{\theta(0)}\theta.$  The next proposition focuses on the cyclic vectors of  $A_v$ and $B_v$.

 \begin{prop}\label{cyclic} Let $A_v:=A_{|h|^2z}^\theta$ be the truncated Toeplitz operator on
 a model space $K_\theta$ with inner function $\theta,$ $\theta(0)=0$ and $v=\la \theta, |h|^2\ra$. Then both $A_v$ and its adjoint $B_v$ are cyclic operators on $K_\theta$,  with cyclic vectors $1$ and $\theta/z$.  \end{prop}
 \begin{proof}
Let $u=\theta/z$. For the operator $B_v,$ we observe that $$B_v u=S^{*2}\theta+u(0)\ol{v}u=S^{*}u+u(0)\ol{v}u,$$ which shows  $\bigvee\{u, B_v u\}$ contains $S^*u$.  By induction, $\bigvee\{B_v^n u: n\geq0\}$ contains $\bigvee\{S^{*n}u:  n\geq0\}$.  Since  $\bigvee\{S^{*n} \theta :\;n\geq 1\}=K_\theta$ (see, e.g. \cite[Proposition 5.15]{GMR}) and $\theta(0)=0,$ we conclude that \begin{align}\bigvee\{B_v^n u:\;n\geq 0 \}=\bigvee\{S^{*n} u :\;n\geq 0\}=K_\theta.\label{dense}\end{align}Similarly, we obtain
\begin{align}\bigvee\{B_v^n 1:\;n\geq 0 \}=\bigvee\{1, S^{*n} u :\;n\geq 0\}=K_\theta.\label{dense1}\end{align}

For the operator $A_v=CB_vC$, where $C$ is the conjugation in \eqref{Con}, we have
 \begin{align*} \bigvee\{A_v^n1:\; n\geq 0\}&=\bigvee\{(CB_vC)^n1:\;n\geq 0 \}\nonumber \\&= \bigvee\{ CB_v^nC1:\;n\geq 0 \}\nonumber\\&=C\bigvee\{B_v^n u:\;n\geq 0 \}. \end{align*}
 This, together with \eqref{dense}, implies $\bigvee\{A_v^n1:\; n\geq 0\}=CK_\theta=K_\theta.$

Likewise, by \eqref{dense1} it yields that
 \begin{align*} \bigvee\{A_v^n u:\; n\geq 0\}&=\bigvee\{(CB_vC)^nu:\;n\geq 0 \}\nonumber \\&= \bigvee\{ CB_v^nCu:\;n\geq 0 \}\nonumber\\&=C\bigvee\{B_v^n 1:\;n\geq 0 \} \\&=C \bigvee\{1,\;B_v^n u:\;n\geq 0 \}=K_\theta.\end{align*} Thus both $1$ and $\theta/z$ are cyclic vectors for $A_v$ and $B_v$ on $K_\theta$ with $\theta(0)=0.$
\end{proof}

A detailed analysis of the operators' actions on kernel functions is conducted to further understand their properties. We will use the notation \begin{align}\widetilde{k_{\lambda}^\theta}=Ck_{\lambda}^\theta=\frac{\theta(z)-\theta(\lambda)}
{z-\lambda}\label{tilde}\end{align} with  the conjugation $C$ given in \eqref{Con}.

\begin{prop}\label{prop ker} Let $A_v:=A_{|h|^2z}^\theta$ be the truncated Toeplitz operator on
a model space $K_\theta$ with  inner function $\theta,$ $\theta(0)=0$ and $v=\la \theta, |h|^2\ra$ and $B_v=A_v^*$.

 $(a)$ For $w\in \mathbb{D}$,
$$B_v k_{w}^\theta=
\theta\bar{z}\ol{v} +\ol{w}k_{w}^{\theta/z},\;\; A_v \widetilde{k_{w}^\theta}= v+ w z\widetilde{k_{w}^{\theta/z}}.$$

$(b)$ For $w\in \mathbb{D}\setminus\{0\},$ $$ A_v k_{w}^\theta =v   \frac{\ol{\theta(w)} }{\ol{w}} +zk_{w}^{\theta/z},\;\;
B_v \widetilde{k_{w}^\theta} =\frac{\theta(w)}{w}\ol{v}\theta\ol{z}+\widetilde{k_w^{\theta/z}}.$$
\end{prop}
\begin{proof}

Let $k_w^\theta$ denote the reproducing kernel for $K_\theta$ at the point $w\in \mathbb{D}$. Now $$k_w^\theta=\la k_w^\theta, 1\ra+k_w^\theta-\la k_w^\theta, 1\ra 1=1+(k_w^\theta-1)$$ where the second term is in the orthogonal complement of $K_z$ in $K_\theta,$ that is, $zK_{\theta/z}.$ Thus \begin{align*} B_vk_w^\theta=\theta\bar{z}\ol{v}+
\bar{z}(k_w^\theta-1),\end{align*} which together with
\begin{align}\bar{z}(k_w^\theta-1) = \ol{w}k_{w}^{\theta/z}\label{equal}\end{align} imply the expression of $B_vk_w^\theta.$

Now let us look at $A_v$, using the fact $(A_v)^*=B_v$, for $\mu, w\in \mathbb{D}$ we have
\begin{align*} \la A_v k_w^\theta, k_\mu^\theta\ra&= \la k_w^\theta, B_vk_\mu^\theta\ra= \la k_w^\theta, \bar{z}\theta\ol{v}+\bar{z}(k_\mu^\theta-1)\ra \\&=v\frac{\ol{\theta(w)}}{\ol{w}}+ \frac{\ol{k_\mu^\theta(w)}-1}{\ol{w}}\\&= v\frac{\ol{\theta(w)}}{\ol{w}}+\frac{k_w^\theta(\mu)-1}
{\ol{w}} \end{align*}
and  by \eqref{equal} we deduce that \begin{align*}A_v k_w^\theta= v \frac{\ol{\theta(w)}}{\ol{w}}+\frac{k_w^\theta-1}
{\ol{w}}=v\frac{\ol{\theta(w)}}{\ol{w}}+zk_w^{\theta/z}. \end{align*}

By the conjugation $C$ in \eqref{Con}, we deduce that

\begin{align*}A_v \widetilde{k_w^\theta}=C B_vk_w^\theta=v+wz\widetilde{k_w^{\theta/z}}. \end{align*}
It further entails that
\begin{align*} B_v\widetilde{k_w^\theta}=C A_v k_w^\theta=\frac{\theta(w)}{w}\ol{v}\theta \bar{z}+\widetilde{k_w^{\theta/z}}. \end{align*}
\end{proof}

\section{Spectra   of the compressed shifts on nearly invariant subspaces}
In this section, we concentrate on determining the spectrum of the compressed shift $A_z^\ma{M}$ on an arbitrary nearly $S^*$-invariant subspace $\ma{M}=hK_\theta$, where  $h$ is an extremal function and $\theta$ is an associated inner function  satisfying $\theta(0)=0$. By virtue of \eqref{SM}, this problem reduces to characterizing the spectrum of the operator $A_v$ on the model space $K_\theta$ with $\theta(0)=0$.  For an inner function \( \theta\in H^2\), it can be expressed as
\[\theta(z) = e^{iy}
\underbrace{z^n
  \left( \prod_{n \geq 1} \frac{a_n}{|a_n|} \frac{a_n - z}{1 - \overline{a_n} z} \right)
}_{B}
\underbrace{
  \exp \left( -\int_{\mathbb{T}} \frac{\xi + z}{\xi - z} d\mu(\xi) \right)
}_{s_\mu}, \]
where \( y \in [0, 2\pi) \), the  factor \( B \) is a Blaschke product with zero set \(\{a_n\}_{n \geq 1}\) (counting multiplicities), and  \( s_\mu \) is a singular inner function with corresponding positive singular measure \( \mu \).  We now define   the spectrum $\sigma(\theta)$
in terms of the above canonical factorization of $\theta$.

\begin{defn}
For a nonconstant inner function \( \theta = Bs_\mu \), the spectrum of \(\theta \) is the set
\[
\sigma(\theta) = \{a_n\}_{n \geq 1}^- \cup \operatorname{supp} \mu.
\]
\end{defn}
In other words, $\sigma(\theta)$  consists of the complement of the set of all points $\lambda\in \ol{\mathbb{D}}$ for which the function $1/\theta$ can be analytically continued into a neighborhood of $\lambda.$ This definition aligns with the standard conception in analytic function theory, where the spectrum is the set of all ``singularities" of the function in question.  Additionally, a useful characterization of $\sigma(\theta)$ in terms of $\liminf$ of zero set of $\theta$ is provided in \cite[Proposition 7.19]{GMR}, which states $$\sigma(\theta)=\{\lambda\in \ol{\mathbb{D}}:\;\lim_{z\in \mathbb{D}}\inf_{z\rightarrow \lambda}|\theta(z)|=0\}.$$

For the compressed shift $S_\theta$ on a model space $K_\theta$, Liv\v{s}ic and Moeller  proved that $\sigma(S_\theta)=\sigma(\theta)$ (see \cite{Li,Mo}). The spectrum $\sigma(T)$ of a linear and bounded operator $T$ is defined as $$\sigma(T)=\{\lambda \in \mathbb{C}:\;T- \lambda I\;\mbox{is not invertible}\}.$$ In addition, a linear bounded operator $T$ is said to be \emph{Fredholm} if its range $R(T)$ is closed and both $R(T)^{\perp}$ and the null space $N(T)$ are finite-dimensional. The essential spectrum is defined as $\sigma_e(T)=\{\lambda\in \mathbb{C}:\; T-\lambda I\;\mbox{is not Fredholm}\}$, which is invariant under perturbation by compact operators (see \cite{SM}). In the case of  $S_\theta$, its essential spectrum is given by $\sigma_e(S_\theta)=\sigma(\theta)\cap \mathbb{T}$, where $\mathbb{T}$ denotes the unit circle (see e.g. \cite[Proposition 9.26]{GMR} or \cite[Lemma 2.5]{Sarason1}). Then Proposition \ref{Aztheta} implies $\sigma_e(A_v)=\sigma(\theta)\cap\mathbb{T}.$ This result, combined with  \eqref{SM}, leads to the following corollary on the essential spectrum of $A_z^\ma{M}$.

\begin{cor}\label{cor ess} Let $A_z^{\ma{M}}$ be the compressed shift on a nearly $S^*$-invariant subspace $\ma{M}=hK_\theta$, where $h$ is an extremal function and $\theta$ is the associated inner function satisfying $\theta(0)=0$;  then $\sigma_e(A_z^\ma{M})=\sigma(\theta)\cap\mathbb{T}.$\end{cor}

Next, the eigenvalues together with eigenvectors and then the whole spectrum are explored for $A_v$ in \eqref{AHZ}
on $K_\theta$ with $\theta(0)=0$ and $v=\la \theta, |h|^2\ra$. The case $v=0$ reduces to the classical compressed shift $S_\theta$ on $K_\theta,$ which is  well understood. We therefore focus on the new setting  $0<|v|<1$. At this time, $A_v$ is a c.n.u. contraction on a model space $K_\theta$ with $\theta(0)=0.$ In consequence, the Sz.-Nagy--Foias theory of characteristic functions for contractions on Hilbert space  becomes an indispensable tool for our subsequent explorations.

Recall from the book \cite{NFBK} that the characteristic function of a contraction $T$ in a Hilbert space $\ma{H}$ is a triple $(\mathfrak{D}_T, \mathfrak{D}_{T^*}, \Theta_T(\lambda))$ where
\begin{align*}
D_T &= \sqrt{I - T^* T}, \quad D_{T^*} = \sqrt{I - T T^*}, \\
\mathfrak{D}_T &= \overline{D_T \mathcal{H}}, \quad \mathfrak{D}_{T^*} = \overline{D_{T^*} \mathcal{H}}  \quad \text{and} \\
\Theta_T(\lambda) &= \left[ -T + \lambda D_{T^*} (1 - \lambda T^*)^{-1} D_T \right] \bigg| \mathfrak{D}_T.
\end{align*} $\mathfrak{D}_T$ and $\mathfrak{D}_{T^*}$ are the defect spaces of $T$.  Let us first proceed with the concrete analysis on the defect  spaces of $A_v$.

\begin{lem}\label{lem defect}  Let $A_v:=A_{|h|^2z}^\theta$ be the truncated Toeplitz operator on  a model space $K_\theta$ with  inner function $\theta,$ $\theta(0)=0$ and $v=\la \theta, |h|^2\ra$, $B_v=A_v^*$.   Then   $\mathfrak{D}_{A_v}= \mathbb{C}\theta/z,\; \mathfrak{D}_{B_v}=\mathbb{C}$. \end{lem}
 \begin{proof}For any $f\in K_\theta,$ there is the decomposition $f=(f-f(0))+f(0)\in zK_{\theta\ol{z}}\oplus K_z$. Combining Propositions \ref{Aztheta} and \ref{Aztheta*}, we deduce that
\begin{align*} A_v B_v f&=A_v B_v (f-f(0)+f(0)) \\&=(f-f(0))+ |v|^2 f(0), \end{align*} which yields that
\begin{align}(I-A_v B_v)f=(1-|v |^2)f(0)=(1-|v|^2)k_0^\theta\otimes k_0^\theta(f)=(1-|v|^2)P_{K_z}(f).\label{Av1}\end{align} Thus we conclude that $\sqrt{I-A_v B_v}=\sqrt{1-|v|^2}P_{K_z}.$
Furthermore, using $CA_v C=B_v$ and \eqref{Av1}, it yields that
\begin{align*} (I-CA_v CC B_vC)f&=(1-|v|^2)Ck_0^\theta \otimes k_0^\theta Cf\\&=(1-| v|^2)\la Cf,k_0^\theta\ra Ck_0^\theta\\&=(1-|v|^2)\la f, \theta/z\ra \theta/z\\&=(1-|v|^2) P_{ \theta/zK_z}f, \end{align*}
which implies that $\sqrt{I- B_v A_v}=\sqrt{1-|v |^2}P_{ \theta/zK_z}.$ Hence the two  defect spaces are $\mathfrak{D}_{A_v}=\mathbb{C} \theta/z$ and $\mathfrak{D}_{B_v} = \mathbb{C}$ with    $\dim \mathfrak{D}_{A_v}=\dim \mathfrak{D}_{B_v}=1.$
\end{proof}

Considering \cite[Theorem 2.3]{NFBK} in the context of our  c.n.u. operator $A_v$, we observe that $A_v$ can be viewed as a special case of $Z(A)$ in \cite[Definition 2.1]{Fu} with $A \xi=v\xi$ for $\xi\in N=\mathbb{C}$. Rather surprisingly, we  use \cite[Theorem 3.2]{Fu} and Lemma \ref{lem defect}  to deduce that $A_v$ on $K_\theta$ with $\theta(0)=0$ has the characteristic function $$\left\{\mathbb{C}\frac{\theta}{z},\;\mathbb{C},\;\Theta_{A_v}(\lambda)
\right\}$$ where $\Theta_{A_v}(\lambda)$ is an analytic operator-valued function whose values are contractive operators from $\mathfrak{D}_{A_v}=\mathbb{C}\frac{\theta}{z} $ into $\mathfrak{D}_{A_v^*}=\mathbb{C}$ and its definition is \begin{align} \Theta_{A_v}(\lambda)\left(\xi\frac{\theta}{z}\right)=
\frac{\theta(\lambda)-v}{1-\ol{v}\theta(\lambda)}\xi,\;~\xi\in \mathbb{C}.\label{Theta}\end{align}

This yields that $A_v$ is unitarily equivalent to the canonical compressed shift $S$ on a new model space  $ H^2(\mathbb{C})\ominus \Theta_{A_v}(\lambda) H^2(\mathbb{C}\theta/z)$, where  $H^2(\mathbb{C})=H^2(\mathbb{D}).$ From \eqref{Theta}, we can write \begin{align*}\Theta_{A_v}(\lambda) H^2(\mathbb{C}\theta/z)&=
\left\{\Theta_{A_v}(\lambda)\left(\frac{\theta}{z}f(z)\right):\;f\in H^2(\mathbb{C})\right\}\\&=\left\{
\frac{\theta(\lambda)-v}{1-\ol{v}\theta(\lambda)}f(z):\;f\in H^2(\mathbb{C})\right\}\\&=
\frac{\theta(\lambda)-v}{1-\ol{v}\theta(\lambda)}H^2(\mathbb{C}).\end{align*} So it follows that  $$ H^2(\mathbb{C})\ominus \Theta_{A_v}(\lambda) H^2(\mathbb{C}\theta/z)=H^2(\mathbb{C})\ominus \frac{\theta(\lambda)-v}{1-\ol{v}\theta(\lambda)}H^2(\mathbb{C})
 =K_{\theta_v}$$ with $\theta_v$  is a Frostman shift of the inner function $\theta$ defined by \begin{align}\theta_v(\lambda)=\frac{\theta(\lambda)-v}
 {1-\ol{v}\theta(\lambda)}.\label{thetav}\end{align}
Therefore, we obtain the unitary model for $A_v$  on $K_\theta$ with $\theta(0)=0$ as below.

\begin{lem} \label{uni-model} Let $A_v:=A_{|h|^2z}^\theta$ be the truncated Toeplitz operator on a model space $K_\theta$ with  inner function $\theta,$ $\theta(0)=0$ and $v=\la \theta, |h|^2\ra$. Then $A_v$ is unitarily equivalent to the compressed shift $S_{\theta_v}$ on a model space $K_{\theta_v}$, where $\theta_v$ is a Frostman shift of $\theta$. \end{lem}

It is shown in \cite{crofoot} (see also  \cite[Proposition 4.3]{CGRW} or \cite[Theorem 13.1]{Sarason1}) that for an inner function \( \theta\) and \(v \in \mathbb{D} \), the Crofoot transform
$$
J_v f = \frac{\sqrt{1 - |v|^2}}{1 - \ol{v}\theta} f
$$
defines a unitary operator from the model space \( K_\theta \) onto the model space \( K_{\theta_v} \) with $\theta_v$ defined in \eqref{thetav}.  Moreover,  its inverse of $J_v^{-1}:\;K_{\theta_v}\rightarrow K_\theta$ is \begin{align}J_v^{-1}g=\frac{\sqrt{1-|v|^2}}{1+\ol{v}\theta_v}g, \;\mbox{for}\;g\in K_{\theta_v}.\label{inver}\end{align}

By the unitary mapping $J_v:\;K_\theta\rightarrow K_{\theta_v},$  the following commutative diagram is derived in \cite[Lemma 13.3]{Sarason1}: that is, $J_vA_v=S_{\theta_v}J_v$.
\begin{align}\begin{CD}
K_{\theta}@> A_v >>K_{\theta}\\
@V J_v  VV @ VV J_v\;V\\
K_{\theta_v}@> S_{\theta_v}>>K_{\theta_v}
\end{CD}\label{AJS} \end{align}

For the compressed shift $S_\theta$ on   $K_\theta$,   \cite[Lemma 2.5]{Sarason1} or \cite[Corollary 9.24]{GMR} exhibits its point spectrum   as \begin{align}\sigma_p(S_\theta)=\sigma(\theta)\cap\mathbb{D}=\{\lambda\in \mathbb{D}:\;\theta(\lambda)=0\}.\label{point}\end{align} Meanwhile, the corresponding eigenvectors are found in \cite[Corollary 9.25]{GMR} by showing \begin{align}\ker(S_\theta-\lambda I)=\mathbb{C}Ck_\lambda^\theta\;\mbox{and}\;\ker(S_\theta^*-\ol{\lambda}I)
=\mathbb{C}k_\lambda^\theta\label{eigek}\end{align}  with conjugation $C$ in \eqref{Con} and $\lambda\in \mathbb{D}$ satisfying $\theta(\lambda)=0.$ So we further proceed to describe the point spectrum and corresponding eigenvectors of $A_v$  on $K_\theta$ with $\theta(0)=0$ by \eqref{AJS}.

  \begin{prop}\label{prop ps}  Let $A_v:=A_{|h|^2 z}^\theta$ be the truncated Toeplitz operator on  a model space $K_\theta$ with inner function $\theta$, $\theta(0)=0$ and $v=\la \theta, |h|^2 \ra$. Then its point spectrum is  given by $$\sigma_p(A_v)=\{\lambda\in \mathbb{D}:\;\theta(\lambda)=v\}.$$ For each $\lambda\in \sigma_p(A_v)$, the corresponding eigenvector is $$Ck_\lambda^\theta=(\theta(z)-v)/(z-\lambda),$$ where $C$ is the conjugation on $K_\theta.$\end{prop}
\begin{proof} By  \eqref{AJS}, we have the unitary equivalence
\begin{align} S_{\theta_v}=J_vA_vJ_v^{-1},\label{asj} \end{align} where $S_{\theta_v}$ denotes the compressed shift on $K_{\theta_v}.$ This immediately yields  $\sigma_p(A_v)=\sigma_p(S_{\theta_v}).$ Applying  \eqref{point} to the Frostman shift $\theta_v$ defined in \eqref{thetav}, we deduce \begin{align*} \sigma_p(S_{\theta_v})=\sigma(\theta_v)\cap \mathbb{D}=\{\lambda\in \mathbb{D}:\;\theta_v(\lambda)=0\}=\{\lambda\in \mathbb{D}:\;\theta(\lambda)=v\}.  \end{align*}
This has also been established in \cite[Theorem 14.1]{Sarason1}.

For the eigenvector, observe that by \eqref{eigek}, we have $S_{\theta_v}(Ck_\lambda^{\theta_v})=\lambda (Ck_\lambda^{\theta_v}).$ Combining this with the unitary equivalence \eqref{asj} yields
\begin{align}  A_v J_v^{-1}(Ck_\lambda^{\theta_v})=\lambda J_v^{-1} (Ck_\lambda^{\theta_v}).\label{jajck} \end{align} Through direct calculations using \eqref{inver}, \eqref{thetav} and \eqref{tilde}, we derive that\begin{align}J_v^{-1}(Ck_\lambda^{\theta_v})
&=J_v^{-1}\left(\frac{\theta_v(z)-0}{z-\lambda}\right)
=\frac{\sqrt{1-|v|^2}}{1+\ol{v}\theta_v}\frac{\theta_v(z) }{z-\lambda}\nonumber\\&=\frac{1}{\sqrt{1-|v|^2}} \frac{\theta-v}{z-\lambda}=\frac{1}{\sqrt{1-|v|^2}}Ck_\lambda^\theta.\label{jvk}
\end{align} Substituting \eqref{jvk} into \eqref{jajck}, we complete the proof.
\end{proof}
Applying $B_v=(A_v)^*$ and $B_v=CA_vC$,   a remark follows on $B_v$'s  point spectrum.

\begin{rem} \label{rem ps}  Let $B_v:=A_{|h|^2 \bar{z}}^\theta$ be the adjoint truncated Toeplitz operator on a model space $K_\theta$ with  inner function $\theta$, $\theta(0)=0$ and $v=\la \theta, |h|^2 \ra$.  Then the point spectrum is given by $$\sigma_p(B_v)=\{\ol{\lambda}\in \mathbb{D}:\;\theta(\lambda)=v\}.$$ For each $\ol{\lambda}\in \sigma_p(B_v)$, the corresponding eigenvector is $k_\lambda^\theta.$ \end{rem}

Two illustrative examples are provided to demonstrate the eigenvalues and eigenvectors of  $A_v$ and $B_v$.
\begin{exm} (1)\;  Let $\theta(z)=z^2\frac{z-a}{1-\ol{a}z}$ when $0<|a|<1.$ Since $K_\theta=K_z\oplus zK_{z\frac{z-a}{1-\ol{a}z}},$ we take the basis  $e_1(z)=1,\;e_2(z)=z/(1-\ol{a}z)$ and $e_3(z)=z^2/(1-\ol{a}z).$
Then it holds that $$A_v e_1=z=e_2-\ol{a}e_3,\;A_ve_2=e_3,\; A_ve_3=A_v(ae_2+\theta/z)=ae_3
+ve_1.$$ So  $A_v$ is represented by the matrix
\begin{align*}[A_v]=\left(\begin{array}{ccc}0 & 0 & v\\
1 & 0& 0\\ -\ol{a} & 1 & a    \end{array}\right).\end{align*}
Its eigenvalues  satisfy $(\lambda-a)\lambda^2-v(1-\lambda\ol{a})=0$ or $\theta(\lambda)=v.$ And let  $[B_v]=[A_v]^*$, we deduce that \begin{align*} ([B_v]-\ol{\lambda}I)k_{\lambda}^\theta
&=\left(\begin{array}{ccc}-\ol{\lambda} & 1 &-{a}\\
0 & -\ol{\lambda}& 1\\ \ol{v} &0 &-\ol{\lambda}+\ol{a}   \end{array}\right) \left(\begin{array}{c}1\\ \frac{\ol{\lambda}}{1-a \ol{\lambda}}\\
\frac{\ol{\lambda}^2}{1-a \ol{\lambda}} \end{array}\right)
 =\left(\begin{array}{c}-\ol{\lambda}+\frac{\ol{\lambda}}{1-a \ol{\lambda}}-\frac{a\ol{\lambda}^2}{1-a \ol{\lambda}}\vspace{1mm}\\ -\frac{\ol{\lambda}^2}{1-a \ol{\lambda}}+\frac{\ol{\lambda}^2}{1-a \ol{\lambda}}\vspace{1mm}\\ \ol{v}+
\frac{\ol{\lambda}^2( \ol{a}-\ol{\lambda})}{1-a \ol{\lambda}} \end{array}\right)=0 \end{align*}
since $\theta(\lambda)=v.$ Thus $k_\lambda^\theta$ is   an eigenvector of $B_v$  on $K_\theta$ with respect to the eigenvalue $\ol{\lambda}$ and then $Ck_\lambda^\theta$ is   an eigenvector of $A_v$  on $K_\theta$ with respect to the eigenvalue $\lambda$.

\vspace{1mm}

(2)\;  Let $\theta(z)=z\left(\frac{z-a}{1-\ol{a}z}\right)^2$ when $0<|a|<1.$ Since $K_\theta= K_{z\frac{z-a}{1-\ol{a}z}}\oplus z\frac{z-a}{1-\ol{a}z}K_{ \frac{z-a}{1-\ol{a}z}} ,$   the basis is  $e_1(z)=1/(1-\ol{a}z),\;e_2(z)=z/(1-\ol{a}z)$ and $e_3(z)=z(z-a)/(1-\ol{a}z)^2.$  Then we deduce that $$A_ve_1=e_2,\;A_ve_2=-\ol{a}ve_1+(a+(\ol{a})^2v)e_2+(1-|a|^2)e_3,\; A_ve_3=ve_1-\ol{a}ve_2+ae_3.$$ That is, $A_v$ behaves as the following matrix
\begin{align*}[A_v]=\left(\begin{array}{ccc}0 &-\ol{a}v & v\\
1 & a+(\ol{a})^2v& -\ol{a}v\\ 0 & 1-|a|^2 & a    \end{array}\right).\end{align*}
Meanwhile, for the adjoint $[B_v]=[A_v]^*$, it holds that\begin{align*} ([B_v]-\ol{\lambda}I)k_{\lambda}^\theta
&=\left(\begin{array}{ccc}-\ol{\lambda} & 1 &0\\
-a\ol{v} & -\ol{\lambda}+\ol{a}+a^2\ol{v}& 1-|a|^2\\ \ol{v} &-a\ol{v} &-\ol{\lambda}+\ol{a}   \end{array}\right) \left(\begin{array}{c} \frac{1}{1-a \ol{\lambda}}\\ \frac{\ol{\lambda}}{1-a \ol{\lambda}}\\
\frac{\ol{\lambda}(\ol{\lambda}-\ol{a})}{(1-a \ol{\lambda})^2} \end{array}\right)\\&=\left(\begin{array}{c}0
\vspace{1mm} \\- a\ol{v} +\ol{\lambda}\frac{\ol{a}-\ol{\lambda}} {1-a \ol{\lambda}}\frac{1-a\ol{\lambda}-1+|a|^2}{1-a \ol{\lambda}} \vspace{1mm}\\ \ol{v} -\ol{\lambda}\frac{(\ol{\lambda}-\ol{a})^2}
{(1-a\ol{\lambda})^2} \end{array}\right)  =\left(\begin{array}{c}0
\vspace{1mm} \\- a\ol{v} + a\ol{\lambda}\frac{(\ol{a}-\ol{\lambda})^2} {(1-a \ol{\lambda})^2} \vspace{1mm}\\0\end{array}\right)= 0 \end{align*}
since  $\theta(\lambda)=v.$ This also exhibits Proposition \ref{prop ps} and Remark \ref{rem ps}.
\end{exm}

Proposition \ref{prop ps}, along with \eqref{SM}, gives the point spectrum and eigenvectors of $A_z^{\ma{M}}$ on $hK_\theta$.
\begin{thm}\label{thm ps}  Let $A_z^\ma{M}$ be the compressed shift on a nearly $S^*$-invariant subspace $\ma{M}=hK_\theta$, where $h$ is an extremal function and $\theta$ is the associated inner function satisfying $\theta(0)=0.$ Then $ \sigma_p(A_z^\ma{M})=\{\lambda\in \mathbb{D}:\;\theta(\lambda)= v \}$ with $v=\la \theta, |h|^2 \ra$. For each $\lambda\in \sigma_p(A_z^\ma{M}),$ the corresponding eigenvector is $h(Ck_\lambda^\theta)=h\widetilde{k_\lambda^\theta}.$\end{thm}

Considering $A_v$ as a one-dimensional perturbation of $S_\theta$, its spectrum must be the union of the  non-Fredholm points of $S_\theta$ and its own eigenvalues, that is,  $$ \sigma(A_v)=\sigma_p(A_v)\cup\sigma_{e}(S_\theta)=\{\lambda\in \mathbb{D}:\;\theta(\lambda)= v \}\cup (\sigma(\theta)\cap\mathbb{T}),$$ which is a special case of \cite[Theorem 14.1]{Sarason1} with the inner function vanishing at zero.  Therefore, the preceding analysis enables us to derive a complete description for the spectrum of $A_z^\ma{M}$ on a nearly $S^*$-invariant subspace $\ma{M}=hK_\theta$.

\begin{thm}\label{thm whole}  Let $A_z^\ma{M}$ be the compressed shift on a nearly $S^*$-invariant subspace $\ma{M}=hK_\theta$, where $h$ is an extremal function and $\theta$ is the associated inner function satisfying $\theta(0)=0.$ Then $$ \sigma(A_z^\ma{M})=\{\lambda\in \mathbb{D}:\;\theta(\lambda)= v \}\cup (\sigma(\theta)\cap\mathbb{T})$$  where $v=\la \theta, |h|^2 \ra$.\end{thm}

\section{Invariant subspaces of the compressed shifts on nearly invariant subspaces}
In this section, we aim to characterize the invariant subspaces of our c.n.u. operator $A_v$ or $B_v$  on  a model space $K_\theta$ with $\theta(0)=0$. As an application, we completely solve {\bf Question 1} by describing all the invariant subspaces of the compressed shift $A_z^\ma{M}$ on arbitrary nearly $S^*$-invariant subspace $\ma{M}=hK_\theta$.

For $f\in K_\theta$ with $\theta(0)=0,$ \eqref{AHZ*} demonstrates $$B_vf =S^*f+\ol{v}  f(0) \theta/z.$$
Write $B_v=S^*+F$ with $F(f)=\ol{v}f(0)\theta/z$ for short.

 Suppose $\ma{N}\subseteq K_\theta$ is an invariant subspace for $B_v$, and that $f\in \ma{N}$ with $f(0)=0$.
Then $S^*f=B_vf-F f=B_vf\in \ma{N}.$ So $\ma{N}$ is a nearly $S^*$-invariant subspace and by Hitt's theorem (Theorem \ref{thm Hitt}) it can be written as $gK_\phi$ with unit norm $g\in \ma{N}$   orthogonal to $\ma{N}\cap zH^2(\mathbb{D})$ and $\phi$ is inner such that $\phi(0)=0.$
So we next  explore when $gK_\phi\subseteq K_\theta$, and when it is invariant under $B_v$.

We note $gK_\phi$ is a nearly $S^*$-invariant subspace, then it is obvious
that $g\in \ma{C}(K_\phi),$ which means that $|g|^2dm$ is a {\em Carleson measure\/} for $K_\phi,$ that is, $g K_\phi\subseteq L^2(\mathbb{T}).$

Motivated by the above questions, we recall  the result  from \cite{FHR} or the more general one in \cite{CaP3}.\vspace{1mm}

\textbf{Theorem A }\cite[Theorem 3.1]{FHR} For inner $\phi$ and $\theta$ and $g\in H^2,$  the following are equivalent:

$(i)$  $gK_\phi\subseteq K_\theta$;

$(ii)$ $g S^*\phi\in  K_\theta,$ and $g\in \mathcal{C}(K_\phi);$

$(iii)$ $g\in \ker T_{\ol{z}\ol{\theta}\phi }\cap\mathcal{C}(K_\phi).$
\vspace{1mm}

Thus Theorem A entails us to have $gK_\phi\subseteq K_\theta$ if and only if $gS^*\phi=g \phi/z\in K_\theta,$ which, since $\phi(0)=0,$ is the same as $g\ol{z}\phi \ol{\theta}\in \ol{zH^2}$, or $\ol{g}\ol{\phi}\theta \in H^2.$

Next we consider $B_v(gK_\phi)\subseteq gK_\phi$. For $k\in K_\phi,$ it holds that
\begin{align*}
B_v(gk)&=S^*(gk)+\ol{v}g(0)k(0)\theta/z\\&=\frac{g(z)k(z)-g(0)k(0)+ \ol{v}g(0)k(0)\theta}{z} \\&
=\frac{g(z)(k(z)-k(0))}{z}+\frac{(g(z)-g(0))k(0)}{z}+\frac{ \ol{v}g(0)k(0)\theta}{z}
\\&=\frac{g(z)(k(z)-k(0))}{z}+k(0)B_vg.
\end{align*}
The first term is always in $gK_\phi,$ and without loss of generality we may take
$k(0)=1$ (for if $k(0)= 0$ there is nothing to prove). So a necessary and
sufficient condition is $B_vg \in gK_\phi. $ In sum, the above steps present  a characterization for invariant subspaces of $B_v$ on $K_\theta$ with $\theta(0)=0.$
\begin{thm}\label{thmiS} Let $\ma{N}$ be an invariant subspace of  $B_v$   on $K_\theta$ with $\theta(0)=0$, then $\ma{N}$ is nearly $S^*$-invariant and then $\ma{N}=gK_\phi$ is an invariant subspace for $B_v$ if and only if $\theta \ol{g\phi}\in H^2$ and $(B_vg)/g \in K_\phi$.\end{thm}

 The finite-dimensional invariant subspaces of $A_v$ or $B_v$ acting on a model space $K_\theta$ exhibit a simpler structure than others. This is particularly evident when the operator  possesses at least one eigenvalue, as the corresponding eigenspace constitutes a nontrivial invariant subspace.

 However, if $\theta$ is a Blaschke product of degree greater than $ 2$ and some eigenvalue has multiplicity $n> 1$, then the generalized eigenspaces also serve as invariant subspaces. More precisely, suppose $\ol{\lambda}$ is an eigenvalue of $B_v$ with multiplicity $n> 1$, then $\ker(B_v-\ol{\lambda})^2$ is  an invariant subspace of $B_v$, which need not coincide with the eigenspace. Consequently,  generalized eigenvectors are necessary to  describe the invariant subspaces, particularly for high-order kernels such as $\ker (B_v-\ol{\lambda}I)^{n+1},\;n\geq 1.$  Analogous results hold for the operator $A_v$.

Now, for a function $ f:\;\mathbb{C}\rightarrow \mathbb{C},$  let $Z(f)$ denote the zero set of $f$, i.e., the collection of all $z\in \mathbb{C}$ such that $f(z)=0$. The following theorem addresses   the location of   $Z(B')$ for a Blaschke product (not necessarily finite) $B$.\vspace{1mm}

\textbf{Theorem B}  \cite[Theorem 2.1]{CC} Let $B$ be a Blaschke product. Then the set $Z(B')\cap \ol{\mathbb{D}}$ is included in the closed convex hull of $Z(B)\cup\{0\}.$\vspace{1mm}

According to Theorem B, when $\theta$ is a Blaschke product of degree $n\geq 2,$ there exists  $v\in \mathbb{D}$   such that the equation $\theta(\lambda)=v$ has a multiple root. Consequently, the eigenvalue $\ol{\lambda}$ of   $B_v$ has algebraic multiplicity greater than $1$. For example, taking $\theta(z)=z\frac{z-\frac{1}{2}}{1-\frac{z}{2}}$ and $v=4\sqrt{3}-7,$ the equation $\theta(\lambda)-v=0$ has a double root $\lambda=2-\sqrt{3},$ which means $B_v$ on $K_\theta$ has an eigenvalue  $2-\sqrt{3}$ with  multiplicity 2.  More generally, for a degree-2 Blaschke product $\theta$,  the eigenvectors of $B_v$ corresponding to a double eigenvalue $\ol{\lambda}$ are precisely scalar multiples of $k_\lambda^\theta.$

\begin{exm}\label{degree2}  Let $\theta(z)=z\frac{z-a}{1-\ol{a}z}$ when $0<|a|<1.$ Write $e_1(z)=1/(1-\ol{a}z)$ and $e_2(z)=z/(1-\ol{a}z).$ So $A_v e_1=e_2$ and  $A_ve_2=ve_1+(a-v\ol{a})e_2.$
Therefore $A_v$ behaves as the matrix
\begin{align*}[A_v]=\left(\begin{array}{cc}0 & v\\
1 & a-v\ol{a}    \end{array}\right).\end{align*}
The eigenvalues satisfy  $\theta(\lambda)=v$ and $\sigma(A_v)=\sigma_p(A_v)=\{\lambda\in \mathbb{D}:\;\theta(\lambda)=v\}.$  Suppose  $x_1e_1+x_2e_2$ is the eigenvector of $B_v$ with respect to $\ol{\lambda}$ and then it holds that
\[
\left( \begin{array}{cc}
\overline{\lambda} & -1 \\
-\overline{v} & \overline{\lambda} - \overline{a} + \overline{v} a
\end{array} \right) \left( \begin{array}{c}
x_1 \\
x_2
\end{array} \right) = 0.
\]
Solving the equations, it yields that the eigenvector   \(x_1e_1+x_2 e_2= tk_{\lambda}^{\theta} \) with $t\in \mathbb{C}$.
\end{exm}
Going beyond  Example \ref{degree2},  let us consider the case where the eigenvalue $\ol{\lambda}$ of $B_v$ has multiplicity $n+1$ for some $n\geq 2$, with $\theta(\lambda)=v$. And then $\theta^{(k)}(\lambda)=0$ for $k=1, \cdots, n.$ We need to formulate the $n$-th derivative kernel $(k_{\lambda}^\theta)^{(n)}$ on $K_\theta$ by the $n$-th derivative kernel in $H^2$ as below
$$k_{\lambda}^{(n)}(z)=\frac{n!z^n}{(1-\ol{\lambda}z)^{n+1}}$$ satisfying $\la f, k_{\lambda}^{(n)}\ra=f^{(n)}(\lambda)$ for all $f\in H^2.$ And then \begin{align*}(k_{\lambda}^\theta)^{(n)}(z)&=\la k_{\lambda}^{(n)}(w),k_{z}^\theta(w)\ra=\la  \frac{n!w^n}{(1-\ol{\lambda}w)^{n+1}}, \frac{1-\ol{\theta(z)}\theta(w)}{1-\ol{z}w}\ra\\&=
\frac{n!z^n}{(1-\ol{\lambda}z)^{n+1}}-\theta(z)\la \frac{n!w^n}{(1-\ol{\lambda}w)^{n+1}}, \frac{\theta(w)}{1-\ol{z}w}\ra\\&=
\frac{n!z^n}{(1-\ol{\lambda}z)^{n+1}}-\theta(z)\sum_{k=0}^nC_n^k
\frac{k!z^{k}}{(1-\ol{\lambda}z)^{k+1}}\ol{\theta^{(n-k)}
(\lambda)}. \end{align*}

Since $\theta(\lambda)=v$ and $\theta^{(k)}(\lambda)=0$ for $k=1,\cdots,n,$ it yields that
\begin{align}(k_\lambda^\theta)^{(n)}(z)=\frac{n!z^n
(1-\ol{v}\theta)}
{(1-\ol{\lambda}z)^{n+1}},\label{nn}\end{align} ensuring $\la f, (k_\lambda^\theta)^{(n)}\ra=f^{(n)}(\lambda)$ for all $f\in K_\theta.$ At this time, the action of $i$-th iteration of $B_v$ is
\begin{align*} B_v^{i}(k_\lambda^\theta)^{(n)}(z)=\frac{n!z^{n-i}(1-\ol{v}\theta)}
{(1-\ol{\lambda}z)^{n+1}},\;\;i=1, 2,\cdots, n. \end{align*} And the action of $n+1$-st iteration of $B_v$ is\begin{align*}B_v^{n+1}(k_\lambda^\theta)^{(n)}(z)&=
n!S^*\left(\frac{ 1-\ol{v}\theta }{(1-\ol{\lambda}z)^{n+1}} \right)+\ol{v}\theta/z\\&=n!\frac{(1-\ol{v}\theta)
(1-(1-\ol{\lambda}z)^{n+1})}{z(1-\ol{\lambda}z)^{n+1}}.\end{align*}
Based on the above calculations, we deduce that \begin{align*}&\quad\quad (B_v-\ol{\lambda}I)^{n+1}(k_\lambda^\theta)^{(n)}(z)\\&=
\sum_{i=0}^{n+1} C_{n+1}^i (-\ol{\lambda})^{n+1-i}B_v^{i}(k_\lambda^\theta)^{(n)}(z) \\&=n!\sum_{i=0}^{n } C_{n+1}^i (-\ol{\lambda})^{n+1-i}\frac{z^{n-i}(1-\ol{v}\theta)}
{(1-\ol{\lambda}z)^{n+1}}+n! \frac{(1-\ol{v}\theta)
(1-(1-\ol{\lambda}z)^{n+1})}{z(1-\ol{\lambda}z)^{n+1}}
\\&= n!\frac{1-\ol{v}\theta}{(1-\ol{\lambda}z)^{n+1}}\left(
-\ol{\lambda}\sum_{i=0}^{n} C_{n+1}^i  (-\ol{\lambda}z)^{n -i}
 + (1-(1-\ol{\lambda}z)^{n+1})/z \right)\\&=n!\frac{1-\ol{v}\theta}{(1-\ol{\lambda}z)^{n+1}} \left(
-\ol{\lambda}\sum_{i=0}^{n} C_{n+1}^i  (-\ol{\lambda}z)^{n -i}
 + \ol{\lambda}\sum_{i=0}^{n } C_{n+1}^i(-\ol{\lambda}z)^{n-i} \right)= 0.\end{align*}
This  implies the $n$-th derivative kernel $(k_{\lambda}^\theta)^{(n)}$ for $\lambda$ in \eqref{nn} is  a generalized eigenvector of $B_v$ when the eigenvalue $\ol{\lambda}$ has multiplicity $n+1$. That is,
\begin{prop}\label{prop ge} Let $B_v:=A_{|h|^2 \bar{z}}^\theta$ be the adjoint truncated Toeplitz operator on a model space $K_\theta$ with  inner function $\theta$, $\theta(0)=0$ and $v=\la \theta, |h|^2 \ra$. Suppose the eigenvalue $\ol{\lambda}$ of $B_v$ has multiplicity $n+1$
for $n\geq 2$ and satisfies $\theta(\lambda)=v,$ then $(k_{\lambda}^\theta)^{(n)}=\frac{n!z^n
(1-\ol{v}\theta)}
{(1-\ol{\lambda}z)^{n+1}}$  is  a generalized eigenvector of $B_v$. \end{prop}

At this stage, we encounter another key question:
\begin{center}
  \textbf{ Question 2}: \textit{How can we characterize the nontrivial invariant subspaces for $A_v$ or $B_v$ on $K_\theta$ with $\theta(0)=0$, particularly when $\theta(\lambda)$ never takes the value $v\in \mathbb{D} $?}
\end{center}

As an example, consider the singular inner function $u(z)=\exp(\frac{z-1}{z+1})$, which satisfies $u(0)=1/e$ and never vanishes.  Defining its Frostman shift of   $u$ as $$\theta(z):= \frac{u(z)-1/e}{1-u(z)/e}=u_{\frac{1}{e}}(z),$$ we observe $\theta$ is inner with $\theta(0)=0$ that never attains the value $v:=-1/e$.  If the spectrum of $A_v$ or $B_v$ is disconnected,   invariant subspaces necessarily exist. In this example, however,  the spectrum consists of a single point on the circle, leaving Question 2 unsolved.

At this stage, the unitary equivalence illustrated in the diagram \eqref{AJS} becomes crucial. Together with the fact that the invariant subspaces of $S_{\theta_v}$ in $K_{\theta_v}$ are precisely those of the form $K_{\theta_v} \ominus K_\phi$ for inner divisors $\phi$ of $\theta_v$ (see, e.g.\cite[Proposition 9.14]{GMR}). This equivalence yields an explicit description of the invariant subspaces for $A_v$, which we now present.

\begin{thm} Let $A_v:=A_{|h|^2 z}^\theta$ be the truncated Toeplitz operator on a  model space $K_\theta$ with  inner function $\theta$, $\theta(0)=0$ and $v=\la \theta, |h|^2 \ra$, then the invariant subspaces of $A_v$ in $K_{\theta}$ are precisely $J_v^{-1}(K_{\theta_v} \ominus K_\phi)$, where $\phi$ divides $\theta_v$ and $J_v^{-1}:\; K_{\theta_v}\rightarrow K_\theta$ is defined in \eqref{inver}.\end{thm}

Finally, using \eqref{SM} we  can completely  answer \textbf{Questions 1} and \textbf{2} with the following theorem.

\begin{thm}\label{thm com} Let $A_z^\ma{M}$ be the compressed shift on a nearly $S^*$-invariant subspace $\ma{M}=hK_\theta$, where $h$ is an extremal function and $\theta$ is the associated inner function satisfying $\theta(0)=0.$ Then every $A_z^\ma{M}$-invariant subspace of $\ma{M}$ is of the form $h J_v^{-1}(K_{\theta_v} \ominus K_\phi)$, where $v=\la \theta, |h|^2 \ra$, $\phi$ is an inner function dividing the Frostman shift $\theta_v$ of $\theta$, and $J_v^{-1}:\; K_{\theta_v}\rightarrow K_\theta$ is the inverse of the Crofoot transform.\end{thm}

Concluding this paper, we present an example to illustrate our main results.

\begin{exm}\label{exm si}Consider the nearly $S^*$-invariant subspace $\ma{M}=(1+z)K_{z^2}$. The unnormalized reproducing kernel at 0, denoted by \(k(z):= P_{\ma{M}}1 \), can be explicitly determined by expressing it in the form \( k(z) = (z+1)(p+qz) \) and doing evaluations at 0 so that
$$\la (z+1)(p+qz), z+1 \ra = 1 \quad \text{and} \quad \la (z+1)(p
+qz), z(z+1) \ra=0.$$ This yields that $ p = 2/3 $ and $q = -1/3$. Now $ \|k\|^{2} = \langle k, k \rangle = k(0) = 2/3$, leading to the normalized  reproducing kernel
\begin{align}
h(z) = \sqrt{\frac{3}{2}} \frac{(2-z)(z+1)}{3} = \frac{(2-z)(z+1)}{\sqrt{6}}.\label{h}\end{align} Now the model space \( K_{\theta} \) must contain \(1/(2-z)\) and \(z/(2-z)\), which enables us to choose
\begin{align}
\theta(z) = z \frac{z - 1/2}{1 - z/2}.\label{theta}\end{align}
So $(1+z)K_{z^2}=hK_\theta$ with  extremal function $h$ in \eqref{h} and  associated inner function $\theta$ in \eqref{theta}.  We then compute
$$v=\la\theta,|h|^2\ra=\la h \theta, h \ra= \frac{1}{6} \left(2(z+1)(z- 1/2)z, (2-z)(z+1) \right)= -\frac{1}{3}.$$
Solving the quadratic equation \( \theta(z) = v \) yields the eigenvalues $$a=\frac{1+i\sqrt{2}}{3},\;\;b=\frac{1-i\sqrt{2}}{3}.$$  By Theorem \ref{thm whole}, the spectrum of $A_z^\ma{M}$ is $$\sigma(A_z^{\ma{M}})=\sigma_p(A_z^\ma{M})=\{a, b\},$$ and Theorem \ref{thm ps} ensures the functions  $ (z-a)(z+1)$ and $(z-b)(z+1)$ are the eigenvectors corresponding to the eigenvalues $b$ and $a$, respectively.

The Frostman shift $\theta_v$  of $\theta$, as defined in  \eqref{thetav}, is given by $$\theta_v(z)=\frac{\theta(z)-v}{1-\ol{v}\theta(z)}
=\frac{3z^2-2z+1}{z^2-2z+3}=\frac{z-a}{1-bz} \frac{z-b}{1-az}:=\theta_1\theta_2, $$ where   $$ \theta_1(z)=\frac{z-a}{1-bz},\;\;\theta_2(z)=\frac{z-b}{1-az}.
\;\;$$   Furthermore, the model space $K_{\theta_v}$ admits two orthogonal decompositions  $$ K_{\theta_v}=K_{\theta_1}\oplus \theta_1K_{\theta_2}=K_{\theta_2}\oplus \theta_2K_{\theta_1}.$$   Theorem \ref{thm com} implies the only nontrivial invariant subspaces of $A_{z}^{\ma{M}}$ on $hK_\theta$ are $hJ_v^{-1}(\theta_1K_{\theta_2})$ and $hJ_v^{-1}(\theta_2 K_{\theta_1})$, where $J_v^{-1}$ is defined in \eqref{inver}. Noting that $$\theta_1K_{\theta_2}=\mathbb{C}\frac{\theta_1}{1-az}\;\;\mbox{and}\;\;
\theta_2K_{\theta_1}=\mathbb{C}\frac{\theta_2}{1-bz},$$  we explicitly calculate  \begin{align*}  J_v^{-1}\left(\frac{\theta_1}{1-az} \right)&=\frac{2\sqrt{2}}{3} \frac{ \theta_1}{1-\frac{1}{3}\theta_v}\frac{1}{1-az}=
\frac{3\sqrt{2}}{2}\frac{z-a}{2-z},\\  J_v^{-1}\left(\frac{\theta_2}{1-bz} \right)&=\frac{2\sqrt{2}}{3} \frac{ \theta_2}{1-\frac{1}{3}\theta_v}\frac{1}{1-bz}=
\frac{3\sqrt{2}}{2}\frac{z-b}{2-z}.\end{align*} Consequently, the  nontrivial invariant subspaces of $A_z^{\ma{M}}$ are $$hJ_v^{-1}(\theta_1 K_{\theta_2})= \mathbb{C}(1+z)(z-a)\;
\;\mbox{and}\;\; hJ_v^{-1}(\theta_2 K_{\theta_1})= \mathbb{C}(1+z)(z-b),$$ which are precisely the eigenspaces  corresponding to the eigenvalues $b$ and $a$, respectively.\end{exm}

 \textbf{Acknowledgements}\;\; The work was supported in part by the National Natural Science Foundation of China (Grant No. 12471126).
\vspace{1mm}

\textbf{Data Availability}\;\; Data sharing not applicable to this article as no datasets were generated or analysed during
the current study.

\textbf{Conflict of interest}\;\; The authors state that there is no conflict of interest.

\end{document}